\documentclass[12pt]{amsart}

\usepackage{amsmath}
\usepackage{amssymb, amsfonts,amscd,verbatim,hyperref,cleveref,graphics}
\usepackage{xspace}
\usepackage{tikz}
\usetikzlibrary{matrix,calc,decorations.pathmorphing,shapes,arrows}
\newtheorem{theorem}{Theorem}[section]
\newtheorem{proposition}[theorem]{Proposition}
\newtheorem{corollary}[theorem]{Corollary}
\newtheorem{lemma}[theorem]{Lemma}
\newtheorem{definition}[theorem]{Definition}
\numberwithin{equation}{section}

\theoremstyle{definition}
\newtheorem{remark}[theorem]{Remark}
\newtheorem{example}[theorem]{Example}

\DeclareSymbolFont{bbold}{U}{bbold}{m}{n}
\DeclareSymbolFontAlphabet{\mathbbold}{bbold}
\def\one{\mathbbold{1}}

\newcommand{\zs}

\newcommand{\term}[1]{{\textit{\textbf{#1}}}}   % To introduce a term
\newcommand{\goesutau}{\xrightarrow{u\tau}}

\newcommand{\goestau}{\xrightarrow{\tau}}	% norm convergence
\definecolor{dred}{RGB}{95,2,31}

\begin{document}

\title[Unbounded topologies]
{Metrizability of minimal and unbounded topologies}

\author{M. Kandi\'c}
\address{Faculty of Mathematics and Physics, University of Ljubljana,
  Jadranska 19, 1000 Ljubljana, Slovenia.}
\email{marko.kandic@fmf.uni-lj.si}

\author{M.A. Taylor}
\address{Department of Mathematical and Statistical Sciences,
         University of Alberta, Edmonton, AB, T6G\,2G1, Canada.}
\email{mataylor@ualberta.ca}

% 46B42 Banach lattices
% 46A40 Ordered topological linear spaces, vector lattices

\thanks{The first author acknowledges financial support from the Slovenian Research Agency (research core funding No. P1-0222). The second author acknowledges support from NSERC and the University of Alberta.}

\date{\today}

\begin{abstract}
In 1987, I.~Labuda proved a general representation theorem that, as a special case, shows that the topology of local convergence in measure is the minimal topology on Orlicz spaces and $L_{\infty}$. Minimal topologies connect with the  recent, and actively studied, subject of ``unbounded convergences". In fact, a Hausdorff locally solid topology $\tau$ on a vector lattice $X$ is minimal iff it is Lebesgue and the $\tau$ and unbounded $\tau$-topologies agree. In this paper, we study metrizability, submetrizability, and local boundedness of the unbounded topology, $u\tau$, associated to $\tau$ on $X$. Regarding metrizability, we prove that if $\tau$ is a locally solid metrizable topology then $u\tau$ is metrizable iff there is a countable set $A$ with $\overline{I(A)}^\tau=X$. We prove that a minimal topology is metrizable iff $X$ has the countable sup property and a countable order basis. In line with the idea that $uo$-convergence generalizes convergence almost everywhere, we prove relations between minimal topologies and $uo$-convergence that generalize classical relations between convergence almost everywhere and convergence in measure.
\end{abstract}

\maketitle

\section{Introduction and preliminaries}

Throughout this paper, $X$ is a vector lattice, assumed Archimedean, and $\tau$ is a locally solid topology on $X$. For a net $(x_{\alpha})$ in $X$, we write $x_{\alpha}\xrightarrow{o}x$ if $(x_{\alpha})$ \textbf{\textit{converges to $x$ in order}}; that is, there is a net $(y_{\beta})$, possibly over a different index set, such that $y_{\beta}\downarrow 0$ and for every $\beta$ there exists $\alpha_0$ such that $|x_{\alpha}-x|\leq y_{\beta}$ whenever $\alpha\geq \alpha_0$. We write $x_{\alpha}\xrightarrow{uo}x$ and say that $(x_{\alpha})$ \textbf{\textit{uo-converges}} to $x\in X$ if $|x_{\alpha}-x|\wedge u\xrightarrow{o}0$ for every $u \in X_+$. Here ``$uo$" stands for ``unbounded order". It is known that if $(\Omega,\Sigma,\mu)$ is a semi-finite measure space and $X$ is a regular sublattice of $L_0(\mu)$, then a sequence $(x_n)$ in $X$ satisfies $x_n\xrightarrow{uo}0$ in $X$ iff $x_n\xrightarrow{uo}0$ in $L_0(\mu)$ iff $x_n\xrightarrow{a.e.}0$, so that $uo$-convergence can be thought of as a generalization of convergence almost everywhere to vector lattices. We refer the reader to \cite{GTX} for further details on $uo$-convergence.

Let $(\Omega,\Sigma,\mu)$ be a $\sigma$-finite measure space. For each $E\in \Sigma$ with $\mu(E)<\infty$ define the Riesz pseudonorm $\rho_E:L_0(\mu)\to \mathbb{R}$ via $$\rho_E(x)=\int_E\frac{|x|}{1+|x|}d\mu.$$ The family of Riesz pseudonorms $\{\rho_E: E\in \Sigma\ \text{and}\ \mu(E)<\infty\}$ defines a Hausdorff locally solid topology $\tau_{\mu}$ on $L_0(\mu)$ known as the \term{topology of (local) convergence in measure on $L_0(\mu)$}. For $0\leq p\leq \infty$, the topology of convergence in measure on $L_p(\mu)$ is defined, simply, as the restriction $\tau_{\mu}|_{L_p(\mu)}$.

In \cite{DOT}, the concept of ``unbounded norm convergence" was introduced as a generalization of convergence in measure. Let $X$ be a Banach lattice. A net $(x_{\alpha})$ in $X$ \textbf{\textit{un-converges}} to $x \in X$ if $|x_{\alpha}-x|\wedge u\xrightarrow{\|\cdot\|}0$ for every $u \in X_+$. The authors show that if $(f_n)$ is a sequence in $L_p(\mu)$ where $1\leq p<\infty$ and $\mu$ is a finite measure, then $f_n\xrightarrow{un}0$ iff $(f_n)$ converges to zero in measure. However, in $L_{\infty}:=L_{\infty}[0,1]$, $un$-convergence agrees with norm convergence and, therefore, fails to agree with convergence in measure.

In \cite{me}, $un$-convergence was further abstracted. Given a locally solid topology $\tau$ on a vector lattice $X$, one can associate a topology, $u\tau$, in the following way. If $\{U_i\}_{i\in I}$ is a base at zero for $\tau$ consisting of solid sets, for each $i\in I$ and $u\in X_+$ define
$$U_{i,u}:=\{x\in X:\; |x|\wedge u\in U_i\}.$$
As was proven in \cite[Theorem 2.3]{me} and \cite{DEM utau}, the collection $\mathcal N_0=\{U_{i,u}:\; i\in I, u\in X_+\}$ is a base of neighbourhoods at zero for a new locally solid topology, denoted by $u\tau$, and referred to as the \term{unbounded $\tau$-topology}. Noting that the map $\tau\mapsto u\tau$ from the set of locally solid topologies on $X$ to itself is idempotent, a locally solid topology $\tau$ is called \term{unbounded} if there is a locally solid topology $\sigma$ with $\tau=u\sigma$ or, equivalently, if $\tau=u\tau.$

A Hausdorff locally solid topology on a vector lattice $X$ is said to be \term{minimal} if there is no coarser Hausdorff locally solid topology on $X$, and \term{least} if it is coarser than every other Hausdorff locally solid topology on $X$. Least topologies were introduced in \cite{AB80} and studied in \cite{AB03}; minimal topologies were studied in \cite{LAB}, \cite{Conradie05}, and \cite{Tay2}.  The following connection between minimal topologies, unbounded topologies, and $uo$-convergence was proven in \cite[Theorem 6.4]{me}. Recall that a locally solid topology $\tau$ is \textbf{\textit{Lebesgue}} if order null nets are $\tau$-null.
\begin{theorem}\label{a.e. implies measure}
Let $\tau$ be a Hausdorff locally solid topology on a vector lattice $X$. TFAE:
\begin{enumerate}
\item\label{a.e. implies measure-1} $uo$-null nets are $\tau$-null;
\item $\tau$ is Lebesgue and unbounded;
\item $\tau$ is minimal.
\end{enumerate}
\end{theorem}
Nets cannot be replaced with sequences in (i) if equivalence is to be maintained. Indeed,  \cite[Theorem 3.9]{GTX} states that order and $uo$-convergences agree for sequences in universally $\sigma$-complete vector lattices. Combining this observation with \cite[Theorem 7.49]{AB03}, we conclude that $uo$-convergent sequences are topologically convergent for any locally solid topology on a universally $\sigma$-complete space. However, \cite[Chapter 7 Exercise 21]{AB03} gives an example of a Hausdorff locally solid topology on a universally $\sigma$-complete vector lattice that fails to be Lebesgue, and thus fails to be minimal.

The equivalence of (i) and (iii) has roots in classical relations between convergence almost everywhere and convergence in measure. Let $(\Omega,\Sigma,\mu)$ be a $\sigma$-finite measure space. It is classically known that for $0\leq p<\infty$, the topology of convergence in measure is the least topology on $L_p(\mu)$, c.f.,  \cite[Theorem 7.55]{AB03} and \cite[Theorem 7.74]{AB03}.  \Cref{a.e. implies measure}(i) reduces to the well-known fact that almost everywhere convergent sequences converge in measure. It can also be used, in conjunction with \cite[Theorem 3.2]{GTX}, to give a one line proof that the restriction of the topology of convergence in measure on $L_0(\mu)$ to any regular sublattice is the minimal topology on said sublattice. We note that, in $L_{\infty}$, the $un$-topology is not minimal. The minimal topology of $L_{\infty}$ is the topology $u|\sigma|(L_{\infty},L_1)$; it agrees with the topology of convergence in measure in $L_{\infty}$. As is shown in \cite[Theorem 7.75]{AB03}, $L_{\infty}$ admits no least topology.
\\

We next recall some notation. Let $A$ be a subset of a vector lattice $X$. The order ideal and the band generated by $A$ are denoted by $I(A)$ and $B(A)$, respectively. If $A=\{a\}$, we define $I_a:=I(\{a\})$ and $B_a:=B(\{a\})$. A positive vector $e\in X$ is said to be a \term{strong unit} if $I_e=X$. If $B_e=X$, then $e$ is called a \term{weak unit}.
If $A$ is at most countable and $B(A)=X$ then, following \cite{Luxemburg71}, we say that $X$ has a \term{countable order basis} (and call  $A$ a countable order basis for $X$).  Obviously, if $e$ is a weak unit in $X$, then $\{e\}$ is a countable order basis for $X$.  A sublattice $Y$ of $X$ is called \term{order dense} if for each $0\neq x\in X_+$ there exists $y\in Y$ with  $0<y\leq x$.

Let $(X,\tau)$ be a topological vector space.  We say that $\tau$ is \term{metrizable} if there exists a metric on $X$ whose metric topology equals $\tau$. We say that $\tau$ is \term{submetrizable} if it is finer than a metrizable topology. A standard fact from topological vector spaces is that a linear topology is metrizable iff it is Hausdorff and first countable (see, e.g., \cite{Kelley:63}). A subset $A$ of $X$ is \term{bounded} if for each neighbourhood $U$ of zero for $\tau$ there exists $\lambda>0$ such that $\lambda A\subseteq U$. If $X$ contains a bounded neighbourhood of zero, then $X$ is said to be \term{locally bounded}. Local boundedness is the strongest of the metrizability related notions. Indeed, if $V$ is a bounded neighbourhood of zero, then a base at zero for $\tau$ is given by $\frac{1}{n} V$ for $n\in \mathbb{N}.$ Hence, every Hausdorff locally bounded topological vector space is first countable and, therefore, metrizable.

A linear topology $\tau$ on a vector lattice $X$ is said to be \term{locally solid} if $\tau$ has a base at zero consisting of solid sets. Note that a locally solid metrizable topology has a countable base at zero consisting of solid sets with trivial intersection. We say that $\tau$ is \term{Riesz submetrizable} if it is finer than a metrizable locally solid topology.

In a Hausdorff locally solid vector lattice, there is an intermediate notion between weak and strong units. Given a positive vector $e$ in a locally solid vector lattice $(X,\tau)$, if $I_e$ is $\tau$-dense in $X$, then $e$ is called a \term{quasi-interior point} of $(X,\tau)$. As in the case of normed lattices, it is easily checked that $e$ is a quasi-interior point iff $x-x\wedge ne\xrightarrow{\tau} 0$ for each $x\in X_+$.

Before we conclude this section we briefly recall the basics regarding the topological completion of a Hausdorff locally solid vector lattice. Let $(X,\tau)$ be a Hausdorff locally solid vector lattice and let $(\widehat X,\widehat \tau)$ be the topological completion of $(X,\tau)$. Then the $\widehat \tau$-closure of $X_+$ in $\widehat X$ is a cone in $\widehat X$ and $(\widehat X,\widehat \tau)$ equipped with this cone is a Hausdorff locally solid vector lattice containing $X$ as a $\widehat \tau$-dense vector sublattice. Moreover, $\widehat \tau$-closures of solid subsets of $X$ are solid in $\widehat X$, and if $\mathcal N_0$ is a base at zero for $(X,\tau)$ consisting of solid sets, then $\{\overline{V}:\; V\in \mathcal N_0\}$ is a base at zero for $(\widehat X,\widehat \tau)$ consisting of solid sets. Here, $\overline V$ denotes the closure of $V$ in $(\widehat X,\widehat \tau)$.  In particular, $(X,\tau)$ is metrizable iff $(\widehat{X},\widehat{\tau})$ is metrizable. For more details on topological vector spaces we refer the reader to \cite{Kelley:63}. All unexplained details in this paper regarding vector lattices and locally solid topologies can be found in \cite{AB03} and \cite{Aliprantis:06}.
\section{Submetrizability of unbounded topologies}
Submetrizability of the unbounded topology was first considered in \cite{KMT}. It is proved in \cite[Proposition 3.3]{KMT} that the unbounded norm topology on a Banach lattice $X$ is submetrizable iff $X$ has a weak unit. It is proved in \cite{DEM utau} that $u\tau$ is submetrizable if $(X,\tau)$ is a metrizable locally solid vector lattice with a weak unit. In \cite{DEM umtau}, the authors proved the converse statement for complete metrizable locally convex-solid vector lattices. In this section, we provide the complete answer on submetrizability of the unbounded topology.

The following example shows that the converse of \cite[Proposition 6]{DEM utau}, in general, does not hold.

\begin{example}\label{c00} 
Consider the vector lattice $c_{00}$ of all eventually null sequences, equipped with the supremum norm. Then $c_{00}$ is a normed lattice without a weak unit, yet the unbounded norm topology is metrizable; a metric $d$ that induces the unbounded norm topology on $c_{00}$ is given by
$$d(x,y)=\sup_n\biggr\{\frac{\min\{|x_n-y_n|,1\}}{n}\biggr\}.$$
\end{example}

It turns out that, when considering submetrizability of the unbounded topology $u\tau$ in spaces that are not complete nor metrizable, the correct replacement  for weak units is the existence of a countable order basis in $X$. Before showing this, we make a remark about countable order bases.
\begin{remark}\label{COB Trick}
It is convenient in the definition of a countable order basis to replace the at most countable set $A$ satisfying $B(A)=X$ with a positive increasing sequence $(u_n)$ satisfying $B(\{u_n\})=X$. This is easily done by enumerating $A=\{a_i\}_{i\in I}$ where $I=\mathbb{N}$ or $\{1,\dots,N\}$ and defining $u_n=|a_1|\vee\cdots\vee |a_n|$ if $n\in I$ and $u_n=u_N$ if $n\in \mathbb{N}\backslash I$. Throughout, when we say that $A=\{u_n\}$ is a countable order basis for $X$ it is tacitly assumed that $(u_n)$ is a positive increasing sequence.
\end{remark}
We also choose to work in more generality. As was shown in \cite[Proposition 9.3]{me}, if $A$ is an ideal of a locally solid vector lattice $(X,\tau)$ and $\{U_i\}_{i\in I}$ is a solid base at zero for $\tau$, then the collection of sets $\{U_{i,a} : i\in I, a\in A_+\}$ defines a locally solid topology $u_A\tau$ on $X$. $u_A\tau$ is known as the \term{unbounded topology on $X$ induced by the ideal $A$}. Note that the topology $u_A\tau$ is Hausdorff iff $\tau$ is Hausdorff and $A$ is order dense in $X$. Also, note that $u_X\tau=u\tau$.
%\begin{remark}\label{COB Trick}
%Let $X$ be a non-trivial vector lattice. In the definition of a countable order basis, the set $A$ satisfying $B(A)=X$ may, WLOG, be chosen to be a positive increasing sequence. Indeed, if not, consider $|A|:=\{|a| : a\in A\}$ and, using that $X$ is non-trivial, choose $0\neq a\in A$, replace $A$ with $B:=|A|\cup \{n|a|\}_{n\in \mathbb{N}}$ and label $B=\{b_n\}$; the $(b_1\vee \cdots\vee b_n)$ satisfies $B(\{b_1\vee\dots\vee b_n\})=X$. These three properties of countable order bases will be assumed, without mention, throughout; the exceptional case of the trivial vector lattice will be tacitly ignored.
%\end{remark}
\begin{proposition}\label{Rmetric} % and metric
Let $(X,\tau)$ be a locally solid vector lattice and $A$ an ideal of $X$.
\begin{enumerate}
\item If $\tau$ is Riesz submetrizable and there is a set in $A$ that is a countable order basis for $X$ then $u_A\tau$ is Riesz submetrizable.
\item If $u_A\tau$ is  submetrizable then there is a set in $A$ that is a countable order basis for $X$.
\end{enumerate}
\end{proposition}
\begin{proof}
(i) Suppose $\{a_n\}\subseteq A_+$ is a countable order basis for $X$. Let $\{U_i\}$ be a countable base at zero of solid sets for a metrizable locally solid topology $\sigma$ coarser than $\tau$. Following the proof of \cite[Theorem 2.3]{me}, one sees that the collection $\{U_{i,a_n}\}$ defines a solid base of neighbourhoods at zero for a locally solid topology $\sigma_1$. This topology is also Hausdorff since if $x\in U_{i,a_n}$ for all $i$ and $n$ then, for fixed $n$, $|x|\wedge a_n \in U_i$ for all $i$ and hence $|x|\wedge a_n=0$ since $\sigma$ is metrizable and hence Hausdorff. By \cite[Lemma 2.2]{LC}, $x=0$. Thus $\sigma_1$ is a locally solid metrizable topology that is clearly coarser than $u_A\tau$.

(ii) Suppose that $u_A\tau$ is submetrizable and let $d$ be a metric that generates a coarser topology than $u_A\tau$. For each $n$, let $B_{\frac{1}{n}}$ be the ball of radius $\frac{1}{n}$ centered at zero for the metric, that is,
\begin{equation}
B_{\frac{1}{n}}=\{x\in X : d(x,0)\leq \tfrac{1}{n}\}.
\end{equation}
Let $\{V_i\}$ be a basis of zero for $\tau$ consisting of solid sets. Since $u_A\tau$ is finer than the metric topology, each $B_{\frac{1}{n}}$ contains some  $V_{i_n,a_n}$ where $0\leq a_n\in A$. Consider $\{a_n\}$. We claim that $B(\{a_n\})=X$. Again, by \cite[Lemma 2.2]{LC}, it suffices to prove that if $x\in X_+$ satisfies $x\wedge a_n=0$ for all $n$ then $x=0$. But $x\wedge a_n=0$ implies that $x \in V_{i_n,a_n}$ and hence $x \in B_{\frac{1}{n}}$ for all $n$. It follows that $x=0$.
\end{proof}

\begin{corollary}\label{submetrizability}
Let $(X,\tau)$ be a locally solid vector lattice and $A$ an ideal of $X$. Then $u_A\tau$ is Riesz submetrizable if and only if $\tau$ is Riesz submetrizable and there is a set in $A$ that is a countable order basis for $X$.
\end{corollary}

Compare with the corresponding result in \cite{KMT}. Note that, in a Banach lattice $X$, a weak unit $e$ can be constructed from a countable order basis $\{e_n\}\subseteq X_+$ via the formula $e:=\sum_{n=1}^{\infty}\frac{1}{2^n}\frac{e_n}{1+\|e_n\|}$. \Cref{submetrizability} also answers a slightly modified version of a question posed on page 14 of \cite{DEM utau}.

\begin{remark}\label{unmet}
Note, in particular, that if $\tau$ is unbounded and Riesz (sub)metrizable then $X$ has a countable order basis.
\end{remark}

\section{Unbounded topologies generated by order ideals}
Let $A$ and $B$ be ideals of a vector lattice $X$, and assume $\tau$ and $\sigma$ are locally solid topologies on $X$. As explained earlier in the paper, and thoroughly in \cite{me}, one can form the topologies $u_A\tau$ and $u_B\sigma$ on $X$. It is then natural to ask how $u_A\tau$ and $u_B\sigma$ relate.  This question was already considered for Banach lattices in \cite{KLT} and was extended in \cite{me}. It has been shown that $u_A\tau=u_B\sigma$ in the following two cases:

 \begin{itemize}
   \item $A$ and $B$ are order dense in $X$ and $\tau$ and $\sigma$ are both Hausdorff Lebesgue topologies on $X$.
   \item $\tau=\sigma$ and $\overline{A}^{\tau}=\overline{B}^{\tau}$.
 \end{itemize}

  In this section, we consider the general case. The results are not only of intrinsic interest, but will be utilized shortly when we characterize metrizability of unbounded topologies.

  Before we state and prove \Cref{un topology by ideals} we need to recall some basic facts on $C(K)$-representations of vector lattices. Suppose $X$ is a vector lattice with a strong unit $u$. For $x\in X$ we define
$$\|x\|_u:=\inf\{\lambda\geq 0:|x|\leq \lambda u\}.$$ It is a standard fact that $\|\cdot\|_u$ defines a lattice norm on $X$, and if $X$ is uniformly complete, then $(X,\|\cdot\|_u)$ is an AM-space with a strong unit $u$. By Kakutani's representation theorem  \cite[Theorem 4.29]{Aliprantis:06}, $(X,\|\cdot\|_u)$ is lattice isometric to some $C(K)$-space for a (unique up to a homeomorphism) compact Hausdorff space $K$. This representation can be taken such that the vector $u$ corresponds to the constant function $\one$. If $X$ is not uniformly complete, consider its order completion $X^\delta$. Then $X$ is an order dense and majorizing sublattice of $X^\delta$. Since order complete vector lattices are uniformly complete and $u$ is also a strong unit for $X^\delta$, by the previous case $X$ is lattice isomorphic to an order dense and majorizing sublattice of $C(K)$ for some compact Hausdorff space $K$.

\begin{theorem}\label{un topology by ideals}
Let $A$ and $B$ be ideals of a vector lattice $X$, and suppose $\tau$ and $\sigma$ are locally solid topologies on $X$. If $u_B\sigma \subseteq u_A\tau$ as topologies on $X$, then $\overline{A\cap B}^{\sigma}=\overline{B}^{\sigma}$.
\end{theorem}

\begin{proof}
It suffices to prove that $B\subseteq \overline{A\cap B}^{\sigma}$. Let $u \in B_+$ and $U$ a solid $\sigma$-neighbourhood of zero. Consider $U_{u}:=\{x\in X: |x|\wedge u\in U\}$. By assumption, there exists $v\in A_+$ and $V$ a solid $\tau$-neighbourhood of zero such that $V_{v}\subseteq U_{u,}$ where $V_{v}:=\{x\in X: |x|\wedge v\in V\}$. This means that for all $x\in X_+$, if $x\wedge v\in V$ then $x\wedge u\in U$.

Let $x_n=(u-nv)^+$. Clearly, $x_n\downarrow$. Put $y_n=x_n\wedge v$. Then $(y_n)\subseteq A_+$, and $y_n\downarrow$.

Consider $I_{u\vee v}$, the ideal generated by $u \vee v$ in $X$. Since $0\leq y_n\leq x_n\leq u\leq u\vee v$, $(x_n)$ and $(y_n)$ are in $I_{u\vee v}$. Also, $u$ and $v$ are in $I_{u\vee v}$. We identify $I_{u\vee v}$ as an order dense majorizing sublattice of $C(K)$ for some $K$ such that $u\vee v$ corresponds to the constant function $\one$.

We next prove that $(y_n)$ converges to zero point-wise in $C(K)$. Take $t\in K$. If $(u\wedge v)(t)=0$ then $0\leq y_n(t)=x_n(t)\wedge v(t)\leq u(t)\wedge v(t)=0.$%$(u\wedge nv)(t)=0$ for all $n$, so that $x_n(t)=u(t)$ and $y_n(t)=0$.

If $(u\wedge v)(t)\neq 0$ then $v(t)>0$, so that $x_n(t)=(u-u\wedge nv)(t)=u(t)-u(t)\wedge nv(t)=0$ for sufficiently large $n$. Thus, for large enough $n$, $y_n(t)=0.$ By Dini's classical theorem $(y_n)$ converges uniformly to zero in $C(K)$. Therefore, for any $N\in \mathbb{N}$ there exists $n_0\in \mathbb{N}$ such that for all $n \geq n_0$ we have $y_n \leq \frac{1}{N}(u\vee v).$ We now go back to $X$. Clearly, $\frac{1}{N}(u\vee v)\xrightarrow{u_A\tau}0$ in $N$ in $X$. Since $u_A\tau$ is locally solid, $y_n\xrightarrow{u_A\tau}0$.

Since $y_n=x_n\wedge v$ is an order bounded sequence in $A$, this implies that $x_n\wedge v\xrightarrow{\tau}0$. Therefore, there exists $m_0$ such that for all $m \geq m_0$, $x_m\wedge v \in V$ and hence $x_m\wedge u\in U$. Since $0 \leq x_m \leq u$, we conclude that for all $m\geq m_0$ we have $x_m \in U$. In particular, $u-u\wedge m_0v\in U$. Since $u \wedge m_0v \in A\cap B$, we conclude $u \in \overline{A\cap B}^{\sigma}$. This proves $B\subseteq \overline{A
\cap B}^{\sigma}$.

\end{proof}

\Cref{un topology by ideals} has many interesting and important consequences. First off, it answers \cite[Question 9.6]{me} affirmatively:

\begin{corollary}\label{closure of ideals give same utau}
Let $A$ and $B$ be ideals of a locally solid vector lattice $(X,\tau)$. Then $u_A\tau=u_B\tau$ iff $\overline{A}^\tau=\overline{B}^\tau$. In particular, $u_A\tau=u\tau$ iff $\overline{A}^{\tau}=X$.
\end{corollary}
It also gives the following corollary that nicely complements \cite[Proposition 9.4]{me}:
\begin{corollary}\label{not utau}
Let $A$ be an ideal of a Hausdorff locally solid vector lattice $(X,\tau)$. Then $u_A\tau$ is Hausdorff and not equal $u\tau$ iff $A$ is order dense but not $\tau$-dense in $X$.
\end{corollary}
\begin{corollary}
Let $(X,\tau)$ be a Hausdorff locally solid vector lattice. Then $\tau$ is minimal iff $\tau=u_A\tau$ for every order dense ideal $A$ of $X$.
\end{corollary}
\begin{proof}
The forward direction is immediate by \cite[Theorem 6.4]{me} combined with \cite[Theorem 9.7]{me}. Suppose $\tau=u_A\tau$ for every order dense ideal $A$ of $X$, but $\tau$ is not minimal. Since $\tau$ is clearly unbounded, it follows from \cite[Theorem 6.4]{me} that $\tau$ is not Lebesgue. By \cite[Theorem 3.8]{AB03}, there exists an order dense ideal $A$ of $X$ that is not $\tau$-dense. By \Cref{closure of ideals give same utau} and the observation that $\tau$ is unbounded, $\tau=u\tau\neq u_A\tau$, a contradiction.
\end{proof}
The next result can be thought of as a topological version of \cite[Lemma 2.2]{LC}. A result of this type was first proved in \cite{DOT}; in that paper, it was shown that a quasi-interior point always witnesses unbounded norm convergence. The converse was proved in \cite[Theorem 3.1]{KMT}; a positive vector is a quasi-interior point iff it witnesses unbounded norm convergence. The \cite{KMT} result was extended (see \cite[Theorem 2]{DEM utau}) to the setting of sequentially complete locally solid vector lattices. The following corollary improves the result of \cite{DEM utau} significantly. Not only does it drop the assumption of sequential completeness, it also characterizes general sets that witness unbounded convergence; they are precisely the sets which generate topologically dense order ideals.

\begin{corollary}\label{Testing}
Let $(X,\tau)$ be a  locally solid vector lattice and $A \subseteq X_+$. TFAE:
\begin{enumerate}
\item $u_{I(A)}\tau=u\tau$;
\item $\overline{I(A)}^{\tau}=X$;
\item For any net $(x_{\alpha})\subseteq X_+$, $x_{\alpha}\xrightarrow{u\tau}0 \Leftrightarrow x_{\alpha}\wedge a\xrightarrow{\tau}0$ for all $a\in A$.
\end{enumerate}
\end{corollary}

\begin{proof}
(i)$\Leftrightarrow$(ii) follows from \Cref{closure of ideals give same utau}. (ii)$\Rightarrow$(iii) is \cite[Proposition 9.9]{me}. (iii)$\Rightarrow$(i) is clear.
\end{proof}
Next we prove that every locally solid vector lattice whose unbounded topology is metrizable admits an at most countable set which generates a dense order ideal.
This result will play an important role in \Cref{metrizability qip} where we consider metrizability of unbounded topologies. Note that we do not assume $\tau$ is metrizable.

\begin{proposition}\label{countably generated closed ideal}
Let $(X,\tau)$ be a locally solid vector lattice such that $u\tau$ is metrizable. Then there exists $e_n\in  X_+$ ($n\in \mathbb{N}$) such that $\overline{I(\{e_n\})}^{\tau}=X$.
\end{proposition}

\begin{proof}
Assume $u\tau$ is metrizable and $\{U_i\}$ is a base at zero consisting of solid (but not even necessarily countably many) sets for $\tau$. Let $d$ be a metric for $u\tau$  and $B_{\frac{1}{n}}$ be the ball at zero of radius $\frac{1}{n}$ for $d$. Then there exists $i_n$ and $e_n \geq 0$ such that \begin{equation}
U_{i_n,e_n} \subseteq B_{\frac{1}{n}}.
\end{equation}
This gives a natural choice of $e_n$ and, indeed, it is straightforward to show that for any net $(x_{\alpha})\subseteq X_+$, $x_{\alpha}\xrightarrow{u\tau}0 \Leftrightarrow x_{\alpha}\wedge e_n\xrightarrow{\tau}0$ for all $n$. This implies that $\overline{I(\{e_n\})}^{\tau}=X$ and that concludes the proof.
\end{proof}

By comparing \Cref{countably generated closed ideal} with \Cref{Rmetric} one should notice that metrizability of $u\tau$ gives the existence of a countable set which generates a topologically dense ideal while submetrizability of $u\tau$ merely gives the existence of a countable order basis.
\\

With \Cref{countably generated closed ideal} in mind, we present the following sequential variant of \Cref{Testing}:
\begin{proposition}\label{qip2}
  Let $(X,\tau)$ be a Hausdorff locally solid vector lattice and $(e_n)$ a positive increasing sequence in $X$. TFAE:
  \begin{enumerate}
  \item\label{qip-un2} $\overline{I(\{e_n\})}^\tau=X$;
  \item\label{qip-seq2} For every sequence $(x_k)$ in $X_+$, if $x_k\wedge
    e_n\goestau 0$ in $k$ for every $n$ then $x_k\goesutau 0$.
  \end{enumerate}
\end{proposition}

\begin{proof}
It suffices to prove (ii)$\Rightarrow$(i):
  Fix $x\in X_+$; we will show that $x\wedge ne_n\goestau x$
 or, equivalently, $(x-ne_n)^+\goestau 0$ as a
  sequence of $n$. Fix $m$ and put $u_m=x\vee e_m$. Now, the ideal $I_{u_m}$ is lattice
  isomorphic (as a vector lattice) to some order dense and majorizing sublattice of $C(K_m)$ for some compact Hausdorff space
  $K_m$, with $u_m$ corresponding to $\one$. Since $x,e_m\in I_{u_m}$, we may
  consider $x$ and $e_m$ as elements of $C(K_m)$. Note that $x\vee e_m=\one$
  implies that $x$ and $e_m$ never vanish simultaneously.

  For each $n\in\mathbb N$, we define
  \begin{displaymath}
    F_n^m=\bigl\{t\in K_m : x(t)\ge ne_m(t)\bigl\}
    \text{ and }
    O_n^m=\bigl\{t\in K_m : x(t)>ne_m(t)\bigl\}.
  \end{displaymath}
  Clearly, $O_n^m\subseteq F_n^m$, $O_n^m$ is open, and $F_n^m$ is closed.

  \emph{Claim 1}: $F_{n+1}^m\subseteq O_n^m$. Indeed, let $t\in
  F_{n+1}^m$. Then $x(t)\ge (n+1)e_m(t)$. If $e_m(t)>0$ then $x(t)>ne_m(t)$,
  so that $t\in O_n^m$. If $e_m(t)=0$ then $x(t)>0$, hence $t\in O_n^m$.

  By Urysohn's Lemma, we find $z_n^{(m)}\in C(K_m)$ such that $0\le z_n^{(m)}\le x$,
  $z_n^{(m)}$ agrees with $x$ on $F_{n+1}^m$ and vanishes outside of $O_n^m$.

  \emph{Claim 2}: $n(z_n^{(m)}\wedge e_m)\le x$. Let $t\in K_m$. If $t\in O_n^m$
  then $n(z_n^{(m)}\wedge e_m)(t)\le ne_m(t)<x(t)$. If $t\notin O_n^m$ then
  $z_n^{(m)}(t)=0$, so that the inequality is satisfied trivially.

  \emph{Claim 3}: $\bigl(x-(n+1)e_m\bigr)^+\le z_n^{(m)}$. Again, let
  $t\in K_m$. If $t\in F_{n+1}^m$ then
  $\bigl(x-(n+1)e_m\bigr)^+\le x(t)=z_n^{(m)}(t)$. If $t\notin F_{n+1}^m$ then
  $x(t)<(n+1)e_m(t)$, so that $\bigl(x-(n+1)e_m\bigr)^+(t)=0$ and the
  inequality is satisfied trivially.

   Denote the vector $\bigl(x-(n+1)e_{n+1}\bigr)^+$ in $X$ by $y_n$. We claim that for each $k$, $y_n\wedge e_k\goestau 0$ in $X$ as a sequence in $n$. Fix $k$ and choose $n\geq k$ arbitrarily.
   Then the following holds in $C(K_{n+1})$:
   $$y_n\wedge e_{n+1}\leq z_n^{(n+1)}\wedge e_{n+1}\leq \tfrac{1}{n}x.$$ In particular, $y_n\wedge e_{n+1}\leq \frac{1}{n}x$ holds in $C(K_{n+1})$ and hence in $X$ since both elements lie in $X$. Recalling that $(e_k)$ is increasing in $X$, we conclude that
   $0 \leq y_n\wedge e_k \leq \frac{1}{n}x$ holds in $X$.
   Since $\tau$ is locally solid, this implies that $y_n\wedge e_k\goestau 0$ for each $k$. The assumption yields that $y_n\goesutau 0$. Since $0\leq y_n\leq x$ we conclude $y_n\goestau 0.$
\end{proof}

\section{Metrizability of unbounded topologies}\label{section: metrizability}

As \Cref{submetrizability} shows, the notion of Riesz submetrizability passes nicely from $\tau$ to $u\tau$ in both directions.
However, the situation with metrizability is not as clean as the following example illustrates.

\begin{example}\label{notmet}
If $X$ is a Banach lattice
then by \cite[Theorem 3.2]{KMT} the
unbounded norm topology is metrizable iff $X$ has a quasi-interior point. Therefore, it should be clear that $\tau$ being metrizable does not guarantee that $u\tau$ is metrizable.

We now provide an example of a nonmetrizable locally solid vector lattice $(X,\tau)$ with a quasi-interior point such that $u\tau$ is  metrizable.

Let $X=L_2:=L_2[0,1]$. Since $X$ is an order continuous Banach lattice,
by \cite[Example 5.6]{me} the unbounded norm topology and unbounded absolute weak topology on $X$ agree. They are metrizable because $X$ has a quasi-interior point. Suppose that the absolute weak topology on $X$ is metrizable. Then by \cite[Theorem 5.6]{AB03}, $L_2$ admits a countable majorizing subset $A=\{f_n\}$. By definition, this means that $I(A)$ is majorizing in $X$, so that $I(A)=X$. Define  $f=\sum_{n=1}^{\infty}\frac{1}{2^n}\frac{|f_n|}{1+\|f_n\|}$. Then $I_f=I(A)=X$, so that $L_2$ has a strong unit. This is a contradiction.
\end{example}
In this section we consider metrizability of the unbounded topology. As was previously mentioned, if  $X$ is a Banach lattice, the unbounded norm topology is metrizable iff $X$ has a quasi-interior point. One direction was extended in \cite{DEM utau} while the complete characterization was obtained in \cite[Proposition 4]{DEM umtau} only for the case of complete metrizable locally convex-solid vector lattices.
In \Cref{metrizability qip} we will provide several improvements to the latter result. We will drop the completeness and local convexity assumptions on $\tau$, replace the existence of a countable topological orthogonal system  with the weaker requirement of a sequence which generates a $\tau$-dense ideal, and prove that metrizability of $u\tau$ is further equivalent to $\widehat X$ possessing a quasi-interior point.

Recall that, by \Cref{countably generated closed ideal}, a necessary condition for $u\tau$ to be metrizable is the existence of an at most countable set $A\subseteq X_+$ with $\overline{I(A)}^\tau=X$. The next example shows that this condition is not sufficient.

\begin{example}
Let $X=\mathbb{R}^J$ where $J$ is an uncountable set. Equipped with the product topology, $\tau$, and point-wise ordering, $X$ is a Hausdorff locally solid vector lattice with the Lebesgue property. It is a standard fact of topology that $(X,\tau)$ is not metrizable. Since the unbounded topology of a product is the product of the unbounded topologies by \cite[Theorem 3.1]{me}, we have $u\tau=\tau$, so that $u\tau$ is not metrizable. Notice that the function $\one \in \mathbb{R}^J$ is a quasi-interior point of $X$ since $\tau$ is Lebesgue and $\one$ is, clearly, a weak unit.
\end{example}

When $\tau$ is metrizable, the following theorem provides the complete answer on metrizability of $u\tau$.

\begin{theorem}\label{metrizability qip}
For a metrizable locally solid vector lattice $(X,\tau)$ the following statements are equivalent:
\begin{enumerate}
  \item There is an at most countable set $A$ in $X$ such that $\overline{I(A)}^\tau=X$;
  \item $u\tau$ is metrizable;
  \item $u\widehat{\tau}$ is metrizable;
  \item The topological completion $\widehat X$ contains a quasi-interior point.
\end{enumerate}
\end{theorem}
\begin{proof}
Recall that $\tau$ is metrizable if and only if $\widehat{\tau}$ is metrizable.

(i)$\Leftrightarrow$(ii): Suppose $(u_n)$ is a positive increasing sequence such that $A=\{u_n\} \subseteq X_+$ satisfies $\overline{I(A)}^{\tau}=X$. Let $\{U_i\}$ be a countable basis at zero for $\tau$ consisting of solid sets. Since, in particular, $B(A)=X$, as in the proof of \Cref{Rmetric} the collection $\{U_{i,u_n}\}$ is a  base at zero for a metrizable locally solid topology $\sigma_1 \subseteq u\tau$. We claim that $\sigma_1=u\tau$. Indeed, by \cite[Proposition 9.5]{me} $u_{I(A)}\tau=u\tau$ and it is easy to see that $\sigma_1=u_{I(A)}\tau$. We already know (ii)$\Rightarrow$(i).

(ii)$\Leftrightarrow$(iii): Suppose $u\tau$ is metrizable. It follows that there is an at most countable set $A$ in $X$ such that $\overline{I(A)}^\tau=X.$ Since $X$ is $\widehat{\tau}$-dense in $(\widehat X,\widehat{\tau})$, $I(A)$ is $\widehat \tau$-dense in $\widehat X$. Hence, the ideal generated by $A$ in $\widehat X$ is also $\widehat{\tau}$-dense in $\widehat X$. This implies that $u\widehat{\tau}$ is metrizable by applying (i)$\Leftrightarrow$(ii) to $\widehat{\tau}$. Conversely, if $u\widehat{\tau}$ is metrizable then so is $u\tau$ since $(u\widehat{\tau})|_X=u\tau$ by \cite[Lemma 3.5]{me}.

(iii)$\Rightarrow$(iv): Since $u\widehat{\tau}$ is metrizable, there exists a sequence $(e_n)\subseteq \widehat{X}_+$ such that $\overline{I(\{e_n\})}^{\widehat{\tau}}=\widehat{X}$. Since $\widehat{\tau}$ is metrizable, there is a countable neighbourhood basis $\{V_n\}$ of zero in $\widehat{X}$ consisting of solid sets such that for each $n\in \mathbb N$ we have $V_{n+1}+V_{n+1}\subseteq V_n.$ For each $n\in \mathbb N$ pick $\lambda_n>0$ such that $\lambda_n e_n\in V_n.$
We claim that the series $\sum_{n=1}^\infty \lambda_n e_n$ converges in $\widehat{X}_+$.
To prove this, define $s_n=\sum_{k=1}^n \lambda_k e_k$ and pick a solid neighbourhood $V_0$ of zero in $\widehat{X}$. Find $n_0\in \mathbb{N}$ such that $V_{n_0}\subseteq V_0.$
Then for $m>n\geq n_0$ we have
$$s_m-s_n=\lambda_{n+1}e_{n+1}+\cdots+\lambda_m e_m \in V_n\subseteq V_{n_0}\subseteq V_0,$$
so that the partial sums $(s_n)$ of the series  $\sum_{n=1}^\infty \lambda_n e_n$ form a Cauchy sequence in $\widehat{X}$.
Since $(\widehat{X},\widehat{\tau})$ is complete and Hausdorff, the series converges to an element of $\widehat{X}_+$. It is clear that $\sum_{n=1}^\infty \lambda_n e_n$ is a quasi-interior point of $\widehat{X}$.

(iv)$\Rightarrow$(iii): Since we have established the implication (i)$\Rightarrow$(ii) for any metrizable locally solid vector lattice, we simply apply it to $(\widehat{X},\widehat{\tau})$.
\end{proof}

In the case when $(X,\tau)$ is complete, \Cref{metrizability qip} reduces to the previously obtained result for Banach lattices.

\begin{corollary}\label{qip = metrizable}
Let $(X,\tau)$ be a complete metrizable locally solid vector lattice. Then $u\tau$ is metrizable iff $X$ has a quasi-interior point.
\end{corollary}

By \Cref{c00} there is no reason to believe that $X$ has a quasi-interior point if $\tau$ and $u\tau$ are metrizable.

\begin{remark}
\Cref{notmet} shows that it can happen that $\tau$ is not metrizable even when $u\tau$ is metrizable and there is a countable set $A$ such that $\overline{I(A)}^{\tau}=X$.
\end{remark}
It so happens that (ii)$\Leftrightarrow$(iii) in \Cref{metrizability qip} remains valid even when $\tau$ is not metrizable. We prove this now:
\begin{proposition}
Let $(X,\tau)$ be a Hausdorff locally solid vector lattice. Then $u\tau$ is metrizable iff $u\widehat{\tau}$ is metrizable.
\end{proposition}
\begin{proof}
If $u\widehat{\tau}$ is metrizable, then so is $u\tau$; this follows since $u\tau=(u\widehat{\tau})|_X$. Suppose $u\tau$ is metrizable and let $\{V_{U_n,u_n}\}$ be a countable basis for $u\tau$ where $U_n$ is a solid $\tau$-closed neighbourhood at zero in $X$, $u_n\in X_+$, and $V_{U_n,u_n}:=\{x\in X: |x|\wedge u_n\in U_n\}$ is defined for notational convenience. Find a $\tau$-closed solid neighbourhood $U_n'$ of zero for $\tau$ with $U_n'+U_n'\subseteq U_n$. We claim that $\{V_{\overline{U_n'}^{\widehat{\tau}}, u_n}\}$ is a basis for $u\widehat{\tau}$ where $V_{\overline{U_n'}^{\widehat{\tau}}, u_n}:=\{\widehat{x}\in \widehat{X}:|\widehat{x}|\wedge u_n \in \overline{U_n'}^{\widehat{\tau}}\}$.

Let $V_{\overline{Z}^{\widehat{\tau}}, \widehat{x}}$ be an arbitrary base neighbourhood of zero for $u\widehat{\tau}$. Here, $Z$ is a solid neighbourhood of zero for $\tau$ and $\widehat{x}\in \widehat{X}_+$. Find $U$ a solid neighbourhood of zero for $\tau$ with $U+U\subseteq Z$. Since $X$ is $\widehat{\tau}$-dense in $\widehat{X}$, there exists $x\in X_+$ with $|\widehat{x}-x|\in \overline{U}^{\widehat{\tau}}$. Find $W$ a solid neighbourhood of zero for $\tau$ with $W+W\subseteq U$. There exists $n$ such that $V_{U_n, u_n}\subseteq V_{W,x}$.

Let $\widehat{y}\in V_{\overline{U_n'}^{\widehat{\tau}}, u_n}$; we will show that $\widehat{y}\in V_{\overline{Z}^{\widehat{\tau}}, \widehat{x}}$. Find $y\in X$ with $|\widehat{y}-y|\in \overline{U_n'}^{\widehat{\tau}}\cap \overline{W}^{\widehat{\tau}}$. Then $$|y|\wedge u_n\leq |\widehat{y}-y|\wedge u_n+|\widehat{y}|\wedge u_n \in \overline{U_n'}^{\widehat{\tau}}+\overline{U_n'}^{\widehat{\tau}}\subseteq \overline{U_n}^{\widehat{\tau}}.$$ Since $U_n$ is $\tau$-closed in $X$, $|y|\wedge u_n\in \overline{U_n}^{\widehat{\tau}}\cap X=U_n.$ Therefore, $y\in V_{U_n,u_n} \subseteq V_{W,x}$. This implies that $|y|\wedge x\in W$. Hence $$|\widehat{y}|\wedge x\leq |\widehat{y}-y|\wedge x+|y|\wedge x\in \overline{W}^{\widehat{\tau}}+W\subseteq \overline{U}^{\widehat{\tau}}.$$ Combining gives, $$|\widehat{y}|\wedge \widehat{x}\leq |\widehat{y}|\wedge |\widehat{x}-x|+|\widehat{y}|\wedge x\in \overline{U}^{\widehat{\tau}}+\overline{U}^{\widehat{\tau}} \subseteq \overline{Z}^{\widehat{\tau}}.$$
\end{proof}

We next extend our results on metrizability to $u_A\tau$. One should compare the next result with \cite[Theorem 3.3]{KLT}:

\begin{proposition}\label{metrizability u_Xtau}
Let $A$ be a $\tau$-closed ideal of a metrizable locally solid vector lattice $(X,\tau)$. TFAE:
\begin{enumerate}
\item $u_A\tau$ on $X$ is metrizable.
\item $u(\tau|_A)$ on $A$ is metrizable and $A$ is order dense in $X$.
\item $A$ contains an at most countable set $B$ such that $\overline{I(B)}^\tau=A$ and $B$ is a countable order basis for $X$.
\end{enumerate}
\end{proposition}
\begin{proof}
(i)$\Rightarrow$(ii): If $u_A\tau$ is metrizable on $X$, then $u(\tau|_A)$ being the relative topology of $u_A\tau$ is metrizable on $A$. Since $u_A\tau$ is Hausdorff, $A$ is order dense in $X$.

(ii)$\Rightarrow$(iii): Since $u(\tau|_A)$ is metrizable, by \Cref{countably generated closed ideal} there is an at most countable set $B\subseteq A_+$ with $\overline{I(B)}^{\tau|_A}=A.$ Since $A$ is closed, this implies that $\overline{I(B)}^{\tau}=A.$ Pick $x\in X_+$ with $x\perp B$. If $x$ is nonzero, there is $a\in A_+$ with $0<a\leq x.$ Since $a\perp B$ and $\overline{I(B)}^{\tau|_A}=A$, we have $a=0$. This contradiction shows $x=0$, so that by \cite[Lemma 2.2]{LC} we conclude that $B$ is a countable order basis for $X$.

(iii)$\Rightarrow$(i): Let $B:=\{b_n\}\subseteq A_+$ be a countable order basis for $X$ such that $\overline{I(B)}^\tau=A$. As always, we assume $(b_n)$ is a positive increasing sequence. Following \Cref{Rmetric}, the sets $U_{i,b_n}:=\{x\in X: |x|\wedge b_n \in U_i\}$, where $\{U_i\}$ is a countable solid base at zero for $\tau$, defines a metrizable locally solid topology $\tau_1$ on $X$. Note that $$x_{\alpha}\xrightarrow{\tau_1}0\Leftrightarrow \forall n \ |x_{\alpha}|\wedge b_n\xrightarrow{\tau}0 \Leftrightarrow x_{\alpha}\xrightarrow{u_{I(B)}\tau}0\Leftrightarrow x_{\alpha}\xrightarrow{u_A\tau}0,$$ so that $u_A\tau$ is metrizable.
\end{proof}
\begin{remark}
The assumption that $A$ is $\tau$-closed is for convenience since $u_A\tau=u_{\overline{A}^\tau}\tau$.
\end{remark}

It is well known that all Hausdorff Lebesgue topologies induce the same topology on order intervals, see, for example, \cite[Theorem 4.22]{AB03}. Since minimal topologies are Hausdorff and Lebesgue, the ``local" properties of minimal topologies are well studied. For example, by \cite[Theorem 4.26]{AB03}, if $\tau$ is a Hausdorff Lebesgue topology then $\tau$ induces a metrizable topology on the order intervals of $X$ if and only if $X$ has the countable sup property. We conclude this section with a complete characterization of ``global" metrizability of minimal topologies.

Let $X$ be a vector lattice admitting a minimal topology $\tau$. In \cite{Conradie05} it was shown that, if $X$ has a weak unit, $\tau$ is metrizable iff $X$ has the countable sup property. \Cref{cool thm} removes the weak unit assumption and  characterizes metrizability of $\tau$ in terms of the vector lattice structure of $X$. Recall that $C_{\tau}$, the carrier of the locally solid topology $\tau$, is defined in \cite[Definition 4.15]{AB03}.
\begin{theorem}\label{cool thm}
Suppose that $X$ is a vector lattice admitting a minimal topology $\tau$. Then $\tau$ is metrizable if and only if $X$ has the countable sup property and a countable order basis.
\end{theorem}
\begin{proof}
Recall that, being minimal, $\tau$ is unbounded, Hausdorff and Lebesgue by \cite[Theorem 6.4]{me}.

If $\tau$ is metrizable and Lebesgue then $X$ has the countable sup property by \cite[Theorem 5.33]{AB03}. Since $\tau$ is metrizable and  locally solid, $\tau$ is Riesz submetrizable. Since $\tau$ is unbounded and Riesz submetrizable, $X$ has a countable order basis by \Cref{unmet}.

Suppose $\tau$ is Hausdorff, unbounded and Lebesgue and that $X$ has the countable sup property and admits a countable order basis. By \cite[Theorem 4.17(b)]{AB03}, $C_{\tau}=X$. Let $\{u_n\}$ be a countable order basis of $X$. As in the proof of \cite[Theorem 4.17(a)]{AB03} there exists a normal sequence $\{U_n\}$ of solid $\tau$-neighbourhoods of zero such that $\{u_n\}\subseteq N^d$, where $N=\bigcap_{n=1}^{\infty}U_n$. Since $N^d$ is a band, $X=B(\{u_n\})\subseteq N^d$, so that $X=N^d$. The sequence $\{U_n\}$ defines a metrizable locally solid topology $\tau'$ on $X$ satisfying $\tau' \subseteq \tau$. Since $\tau$ is minimal, $\tau=\tau'$ so that $\tau$ is metrizable.
\end{proof}

 We remark that $C[0,1]$ satisfies the countable sup property, admits a metrizable locally solid topology, and has a countable order basis, but does not admit a minimal topology. Indeed, a vector lattice admits a minimal topology iff it admits a Hausdorff Lebesgue topology, and no Hausdorff Lebesgue topology exists on $C[0,1]$ by \cite[Example 3.2]{AB03}.
\\

Recall that \cite[Theorem 7.55]{AB03} states that if a laterally $\sigma$-complete vector lattice admits a metrizable locally solid topology $\tau$ then $\tau$ is the only Hausdorff locally solid topology on $X$ and $\tau$ is Lebesgue. Therefore, $\tau$ is the minimal topology. After reminding oneself that laterally complete vector lattices admit weak units by, say, \cite[Theorem 7.2]{AB03}, the next corollary is an immediate consequence of \cite[Corollary 5.3]{me}: 

\begin{corollary}
Suppose $X$ is a vector lattice that admits a minimal topology $\tau$. $\tau$ is complete and metrizable if and only if $X$ is universally complete and has the countable sup property. In this case, $\tau$ is the only Hausdorff locally solid topology on $X$.
\end{corollary}

\section{Locally bounded unbounded topologies}

In this section we present a theorem and examples regarding local boundedness of unbounded topologies. If $(X,\tau)$ is locally bounded and Hausdorff, then $(X,\tau)$ is metrizable. By \Cref{metrizability qip} we already know that metrizability of $u\tau$ is equivalent to the topological completion $\widehat X$ of $X$ having a quasi-interior point. When studying Hausdorff locally bounded unbounded topologies, it is strong units of the topological completion that are of interest.

\begin{theorem}\label{answer to LB}
Let $\tau$ be a Hausdorff locally solid topology on a vector lattice $X$. TFAE: 
\begin{enumerate}
\item $u\tau$ is locally bounded;
\item $u\tau$ has an order bounded neighbourhood of zero;
\item $\tau$ has an order bounded neighbourhood of zero;
\item $\widehat{X}$ has a strong unit and $\tau$ is metrizable;
\item $\widehat{\tau}$ coincides with the $\|\cdot\|_u$-topology, where $u$ is a strong unit of $\widehat{X}$, and $(\widehat{X},\|\cdot \|_u)$ is a Banach lattice which is lattice isometric to a $C(K)$-space;
\item $u\widehat{\tau}$ is locally bounded.
\end{enumerate}
In this case, $\tau=u\tau$ and $\widehat{\tau}=u\widehat{\tau}$.
\end{theorem}
\begin{proof}
(i)$\Rightarrow$(ii): Suppose $u\tau$ is locally bounded. Then there exists a neighbourhood $V$ of zero such that a base at zero for $u\tau$ is given by $\varepsilon V$ for $\varepsilon>0$. Find $U$ a solid neighbourhood of zero for $\tau$ and $u \in X_+$ so that, in the notation of \cite[Lemma 2.16]{me}, $U_u \subseteq V$. We claim that $V$ cannot contain a non-trivial ideal, so that $U_u$ cannot contain a non-trivial ideal, so that $U_u \subseteq [-u,u]$ by \cite[Lemma 2.16]{me}. Suppose $V$ contains a non-trivial ideal. Then there exists $x\neq 0$ such that $\lambda x\in V$ for all $\lambda>0$. However, this implies that $x\in \varepsilon V$ for all $\varepsilon>0$ and hence $x=0$. This is a contradiction.

(ii)$\Rightarrow$(iii) because $u\tau \subseteq \tau$.

(iii)$\Rightarrow$(iv): Suppose $\tau$ has an order bounded neighbourhood of zero, say, $[-u,u]_X$, where the subscript denotes the space in which the order interval is taken. It follows that $\tau$ is metrizable and $\overline{[-u,u]_X}^{\widehat{\tau}}$ is a $\widehat{\tau}$-neighbourhood of zero; it is contained in $[-u,u]_{\widehat{X}}$ since the cone $\widehat{X}_+$ is $\widehat{\tau}$-closed. Therefore, $\widehat{\tau}$ has an order bounded neighbourhood of zero and thus $\widehat{X}$ has a strong unit.

(iv)$\Rightarrow$(v): Let $u$ be a strong unit for $\widehat{X}$. Since $\tau$ is metrizable, $\widehat{\tau}$ is complete and metrizable. It is easy to see that $\widehat{X}$ is uniformly complete so that, by \cite[Theorem 5.21]{AB03}, the $\widehat{\tau}$ and $\|\cdot\|_u$-topologies agree. That $(\widehat{X},\|\cdot\|_u)$ is lattice isometric to a $C(K)$-space follows from Kakutani's representation theorem.

(v)$\Rightarrow$(vi): It follows from \cite[Theorem 2.3]{KMT} that $\widehat{\tau}=u\widehat{\tau}$ so that $u\widehat{\tau}$ is locally bounded.

(vi)$\Rightarrow$(i) follows since $(u\widehat{\tau})|_X=u\tau$. For the additional clause, it has been shown that $\widehat{\tau}=u\widehat{\tau}$ from which it follows that $\tau=(u\widehat{\tau})|_X=u(\widehat{\tau}|_X)=u\tau$.
\end{proof}
Clearly, \Cref{answer to LB} implies \cite[Theorem 2.3]{KMT}. We next show that our results cannot be extended to $u_A\tau$:
\begin{proposition}
Let $A$ be an order dense ideal of a Hausdorff locally solid vector lattice $(X,\tau)$. If $u_A\tau$ is locally bounded then $A=X$.
\end{proposition}
\begin{proof}
Notice first that $u_A\tau$ is Hausdorff. Assuming $u_A\tau$ is locally bounded, there exists a neighbourhood $V$ such that a base at zero for $u_A\tau$ is given by $\varepsilon V$ for $\varepsilon>0$. Find $U$ a solid neighbourhood of zero for $\tau$ and $a \in A_+$ so that, in the notation of \cite[Lemma 2.16]{me}, $U_a \subseteq V$. As in the proof of \Cref{answer to LB}, $U_a \subseteq [-a,a]$. Since neighbourhoods are absorbing, $a$ is a strong unit for $X$. Therefore, $X=I_a\subseteq A$.
\end{proof}
\begin{example}
Consider $(X_1,\tau_1):=(C[0,1],\|\cdot\|_{\infty})$. $X_1$ is a complete, Hausdorff, locally bounded, unbounded locally solid vector lattice that has a strong unit. On the other hand, $(X_2,\tau_2):=(C[0,1],\|\cdot\|_2)$ is a Hausdorff, locally bounded, locally solid topology, but the topology $u\tau_2$ is not locally bounded. %Indeed, it agrees with the topology of convergence in measure on $L_2$ restricted to $C[0,1]$.
This is consistent with \Cref{answer to LB} as $\widehat{(X_2,\tau_2)}=(L_2,\|\cdot\|_2)$ does not have a strong unit.
\end{example}
\begin{corollary}
Let $X$ be a vector lattice admitting a minimal topology $\tau$. TFAE:
\begin{enumerate}
\item $\tau$ is locally bounded;
\item $X$ is finite dimensional.
\end{enumerate}
\end{corollary}
\begin{proof}
(i)$\Rightarrow$(ii): Since $\tau$ is minimal, it is Hausdorff, and $\tau=u\tau$. \Cref{answer to LB} implies that $\widehat{X}$ has a strong unit. By \cite[Theorem 5.2]{me}, $X^u$ has a strong unit. By \cite[Theorem 7.47]{AB03}, $X$ is finite dimensional. The other direction is clear.
\end{proof}

\section{Measure-theoretic results}
In this section we investigate relations between minimal topologies and $uo$-convergence. As an application, we deduce classical results in Measure theory using our measure-free language. We begin with a definition:

\begin{definition}
Let $X$ be a vector lattice. We say that $uo$-convergence on $X$ is \term{sequential} if whenever $(x_{\alpha})_{\alpha\in A}$ is a $uo$-null net in $X$ there exists an increasing sequence of indices $\alpha_n\in A$ such that $x_{\alpha_n}\xrightarrow{uo}0$.
\end{definition}

\begin{theorem}\label{COBCSP}
Let $X$ be a vector lattice. TFAE:
\begin{enumerate}
\item $X$ has a countable order basis and the countable sup property;
\item $X^u$ has the countable sup property;
\item $uo$-convergence on $X^u$ is sequential.
\end{enumerate}
Moreover, in this case, $uo$-convergence on $X$ is sequential.
\end{theorem}
\begin{proof}
(i)$\Rightarrow$(ii): Let $\{u_n\}\subseteq X_+$ be a countable order basis for $X$. It follows from order density of $X$ in $X^{\delta}$ that $\{u_n\}$ is also a countable order basis of $X^{\delta}$. It is known that $X^{\delta}$ inherits the countable sup property from $X$: see \cite[Lemma 1.44]{AB03}. Therefore, since $X$ and $X^{\delta}$ have the same universal completion, we may assume, by passing to $X^{\delta}$, that $X$ is order complete. This implies that $X$ is an ideal of $X^u$. Since $X$ is order dense in $X^u$, $\{u_n\}$ is a countable order basis of $X^u$.

Let $A$ be a non-empty disjoint  subset of $X^u_+$. For each $n$, the set $A_n:=A\wedge u_n$ is a non-empty order bounded disjoint subset of $X_+$. Since $X$ has the countable sup property, for each $n$, $A_n$ is at most countable. Hence, taking into account that $A$ is disjoint, the set of all $a\in A$ such that $a \wedge u_n\neq 0$ for some $n$ is at most countable. Since $\{u_n\}$ is a countable order basis for $X^u$, we conclude that at most countably many $a\in A$ are non-zero. 
\cite[Exercise 1.15]{AB03} yields that $X^u$ has the countable sup property.

(ii)$\Rightarrow$(i): It is clear that $X$ inherits the countable sup property from $X^u$, so we show that $X$ has a countable order basis. Since $X^u$ is universally complete, it has a weak unit $e$. Since $X$ is order dense in $X^u$, there exists a net $(x_{\alpha})$ in $X_+$ such that $x_{\alpha}\uparrow e$. By the countable sup property there is an increasing sequence $(\alpha_n)$ such that $x_{\alpha_n}\uparrow e$. It is easy to see that $(x_{\alpha_n})$ is a countable order basis of $X$.

(ii)$\Leftrightarrow$(iii): It is clear that if $uo$-convergence is sequential on $X^u$ then $X^u$ has the countable sup property. For the converse, let $e$ be a weak unit of $X^u$. It is both known and easy to check that a vector lattice $X$ has the countable sup property iff order convergence is sequential. Using this, we conclude that for a net $(x_{\alpha})$ in $X^u$, $$x_{\alpha}\xrightarrow{uo}0\Leftrightarrow |x_{\alpha}|\wedge e\xrightarrow{o}0\Rightarrow \exists \alpha_n \ |x_{\alpha_n}|\wedge e\xrightarrow{o}0\Leftrightarrow x_{\alpha_n}\xrightarrow{uo}0.$$
For the moreover clause, let $(x_{\alpha})$ be a net in $X$ such that $x_{\alpha}\xrightarrow{uo}0$ in $X$. Since $X$ is regular in $X^u$, $x_{\alpha}\xrightarrow{uo}0$ in $X^u$. Since $uo$-convergence on $X^u$ is sequential, there is an increasing sequence $(\alpha_n)$ of indices such that $x_{\alpha_n}\xrightarrow{uo}0$ in $X^u$, hence in $X$.
\end{proof}
\begin{remark}
If $uo$-convergence on $X$ is sequential then, of course, $X$ has the countable sup property. However, it does not follow that $X$ has a countable order basis. Indeed, let $X$ be an order continuous Banach lattice. By \cite[Corollary 3.5]{DOT}, $uo$-convergence is sequential. Note that, in Banach lattices, admitting a countable order basis is the same as admitting a weak unit, and not all order continuous Banach lattices admit weak units.
\end{remark}
%I believe there is a uo analogue of \cite[Theorem 4.19]{AB03} for uo-convergence. I imagine it would look something like this.
\Cref{COBCSP} completes \cite[Lemma 2.9]{LC}. Combining with \Cref{cool thm} we get:
\begin{corollary}\label{Met CSP}
Let $X$ be a vector lattice admitting a minimal topology $\tau$. Then $\tau$ is metrizable if and only if $X^u$ has the countable sup property.
\end{corollary}

Recall that a measure space $(\Omega,\Sigma,\mu)$ is  \term{semi-finite} if whenever $E\in \Sigma$ and $\mu(E)=\infty$ there is an $F\subseteq E$ such that $F\in \Sigma$ and $0<\mu(F)<\infty$. Semi-finiteness of the measure is equivalent to the topology of convergence in measure on $L_0(\mu)$ being Hausdorff (see e.g. \cite[Theorem 245E]{Fre}).

\begin{proposition}\label{CSP==sigma-finite}
Let $(\Omega, \Sigma,\mu)$ be a semi-finite measure space. Then $L_0(\mu)$ has the countable sup property iff $\mu$ is $\sigma$-finite.
\end{proposition}

\begin{proof}
If $\mu$ is $\sigma$-finite, then $L_0(\mu)$ has the countable sup property by \cite[Theorem 7.73]{AB03}.

For the converse, assume $L_0(\mu)$ has the countable sup property. Let $\mathcal C$ be the collection of all families of pairwise disjoint measurable sets of finite non-zero measure. The family $\mathcal C$ is partially ordered by inclusion. If $\mathcal C_0$ is a chain in $\mathcal C$, then the union of the chain is an upper bound for $\mathcal C_0$ in $\mathcal C$. Hence, by Zorn's lemma there is a maximal family $\mathcal F$ of pairwise disjoint measurable sets of finite non-zero measure. The set of functions $\{\chi_F:\; F\in \mathcal F\}$ is bounded above by $\one$ in $L_0(\mu)$. Let $E$ be the union of all sets in $\mathcal F$. Since $L_0(\mu)$ has the countable sup property, \cite[Exercise 1.15]{AB03} implies $\mathcal F$ is at most countable, so that $E$ is measurable. If $\mu(X\setminus E)>0$, since $\mu$ is semi-finite, there is a measurable subset $E'\subseteq X\setminus E$ with $0<\mu(E')<\infty.$ This contradicts maximality of $\mathcal F$. Hence, $\mu(X\setminus E)=0$ and $\mu$ is $\sigma$-finite.
\end{proof}
\begin{remark}
By \cite[Theorem 7.73]{AB03}, if $\mu$ is a $\sigma$-finite measure then for $0\leq p\leq \infty$, $L_0(\mu)$ is the universal completion of $L_p(\mu)$ and, moreover, $L_0(\mu)$ has the countable sup property.  In this case, \Cref{Met CSP} simply states that convergence in measure in $L_p(\mu)$ is metrizable. This is consistent with classical results. Indeed, it is known that the topology of convergence in measure on $L_0(\mu)$ is metrizable iff $\mu$ is $\sigma$-finite (see e.g. \cite[Theorem 245E]{Fre}). \Cref{CSP==sigma-finite} suggests  using the countable sup property as a replacement for $\sigma$-finiteness in general vector lattices.
\end{remark}

It is known (see \cite[Exercise 5.8]{AB03}) that if $\tau$ is a complete metrizable locally solid topology then one can extract order convergent subsequences from $\tau$-convergent sequences. Note that, by \cite[Theorem 5.20]{AB03}, $\tau$ is the greatest locally solid topology on $X$ in the sense that if $\tau'$ is a locally solid topology on $X$ then $\tau' \subseteq \tau$. The next theorem shows that it is often possible to extract $uo$-convergent sequences from nets that converge in significantly weaker topologies.

\begin{theorem}\label{measure to a.e.}
Let $(X,\tau)$ be a Hausdorff locally solid vector lattice with the Fatou property. Assume that $C_{\tau}$ has a countable order basis and let $(x_{\alpha})_{\alpha\in A}$ be a net in $X$. If $x_{\alpha}\xrightarrow{\tau}0$ in $X$ then there exists an
increasing sequence of indices $\alpha_n\in A$ such that $x_{\alpha_n}\xrightarrow{uo}0$ in $X$.
\end{theorem}

\begin{proof}
Let $\{u_k\} \subseteq (C_{\tau})_+$ be a countable order basis for $C_{\tau}$.  By \cite[Theorem 4.12]{AB03}, $\tau$ extends uniquely to a Fatou topology $\tau^{\delta}$ on $X^{\delta}$. By \cite[Exercise 4.5]{AB03}, $\{u_k\}\subseteq C_{\tau^{\delta}}$. \cite[Theorem 4.17]{AB03} tells us that $\{u_k\}$ is a countable order basis for $C_{\tau^{\delta}}$ and $C_{\tau^{\delta}}$ is an order dense ideal of $X^{\delta}$.

Assume that a net $(x_{\alpha})$ in $X$ satisfies $x_{\alpha}\xrightarrow{\tau}0$. As in the proof of \cite[Theorem 4.17]{AB03}, choose a normal sequence $\{V_n\}$ of Fatou  $\tau^{\delta}$-neighbourhoods of zero such that $\{u_k\}\subseteq N^d$ where $N=\bigcap_{n=1}^{\infty} V_n$. Since $(x_{\alpha})$ is a $\tau^{\delta}$-Cauchy net of $X^{\delta}$, there is an increasing sequence $(\alpha_n)$ of indices such that $x_{\alpha_{n+1}}-x_{\alpha_n}\in V_{n+2}$ and $x_{\alpha_n}\in V_n$ for all $n$. This implies that for each $k$ and $n$, $|x_{\alpha_{n+1}}|\wedge u_k-|x_{\alpha_n}|\wedge u_k \in V_{n+2}$ and $|x_{\alpha_n}|\wedge u_k\in V_n$. Put $v^{*k}=\limsup_n |x_{\alpha_n}|\wedge u_k$ and $w^{*k}=\liminf_n |x_{\alpha_n}|\wedge u_k$ in $X^{\delta}$. By \cite[Lemma 4.14]{AB03}, $|x_{\alpha_n}|\wedge u_k- v^{*k}\in V_n$ for all $n$ and all $k$. Therefore, $v^{*k}\in V_n$ for each $n$. We conclude that $v^{*k}\in N\cap N^d$, and, therefore, $v^{*k}=0$. Similarly, $w^{*k}=0$. This implies that $|x_{\alpha_n}|\wedge u_k\xrightarrow{o}0$ in $n$ in $X^{\delta}$. Since $\{u_k\}$ is a countable order basis of $C_{\tau^{\delta}}$, an order dense ideal of $X^{\delta}$, $\{u_k\}$ is a countable order basis of $X^{\delta}$. This implies that $x_{\alpha_n}\xrightarrow{uo}0$ in $X^{\delta}$, hence in $X$.
\end{proof}

We need the following lemma in order to establish our final result:

\begin{lemma}
Suppose $X$ is a vector lattice admitting a Hausdorff Lebesgue topology $\tau$. TFAE:
\begin{enumerate}
\item $C_{\tau}$ has a countable order basis;
\item $X$ has the countable sup property and a countable order basis;
\item $X^u$ has the countable sup property.
\end{enumerate}
In this case, $C_{\tau}=X$.
\end{lemma}
\begin{proof}
Only (i)$\Leftrightarrow$(ii) requires proof.

Suppose that $X$ has the countable sup property and a countable order basis. By \cite[Theorem 4.17]{AB03}, $C_{\tau}=X$ and, therefore, $C_{\tau}$ has a countable order basis. For the converse, assume that $C_{\tau}$ has a countable order basis; denote it by $\{u_k\}\subseteq (C_{\tau})_+$. As in the proof of \cite[Theorem 4.17]{AB03}, there is a normal sequence $\{U_n\}$ of solid $\tau$-neighbourhoods of zero such that $\{u_k\}\subseteq N^d\subseteq C_{\tau}$ where $N=\bigcap_{n=1}^{\infty} U_n$, and the disjoint complement is taken in $X$. Since $C_{\tau}$ is an ideal of $X$, $N^d$ is a band of $C_{\tau}$. Since $\{u_k\}\subseteq N^d\subseteq C_{\tau}$ and $\{u_k\}$ is a countable order basis for $C_{\tau}$, $N^d=C_{\tau}$. This implies, by \cite[Theorem 4.17]{AB03}, that $C_{\tau}$ is an order dense band of $X$ and, therefore, $C_{\tau}=X$. Since $\tau$ is Lebesgue, $C_{\tau}=X$ has the countable sup property.
\end{proof}

The final corollary generalizes another classical relation between a.e.~convergence and convergence in measure. As noted, the countable sup property assumption acts as a replacement for $\sigma$-finiteness in general vector lattices.

\begin{corollary}\label{further}
Let $X$ be a vector lattice admitting a minimal topology $\tau$. Assume that $X^u$ has the countable sup property. Then a sequence $(x_n)$ in $X$ is $\tau$-convergent to zero in $X$ if and only if every subsequence of $(x_n)$ has a further subsequence that $uo$-converges to zero in $X$.
\end{corollary}
\begin{proof}
The result follows immediately by combining \Cref{measure to a.e.} with \cite[Proposition 2.22]{me}.
\end{proof}
\begin{example}
The assumption that $X^u$ has the countable sup property is crucial when trying to extract $uo$-convergent subsequences from topologically convergent sequences. Indeed, consider \cite[Example 9.6]{KLT}. This gives an example of an order continuous Banach lattice $X$ without a weak unit such that the minimal topology on the universal completion (which is the $un$-topology on $X^u$ induced by $X$; it is locally solid by \cite[Theorem 9.11]{me}) has a null sequence with no $uo$-null subsequences. Note that order continuous Banach lattices have the countable sup property and admit a countable order basis iff they admit a weak unit.
\end{example}
\subsection*{Acknowledgements} The second author would like to thank Vladimir Troitsky and Niushan Gao for valuable discussions.

\end{document}